\documentclass[12pt]{amsart}
\usepackage{amsfonts}
\usepackage{amsmath}
\usepackage{amssymb,latexsym}
\usepackage[mathcal]{eucal}
\usepackage{amscd}

\usepackage{amsxtra}

\input xy
\xyoption{all}

\newcommand{\ch}{\mathbf{1}}

\newcommand{\Z}{\mathbb{Z}}

\newcommand{\N}{\mathbb{N}}

\newcommand{\al}{\alpha}

\newcommand{\del}{\delta}

\newcommand{\ep}{\epsilon}
\newcommand{\sig}{\sigma}

\newcommand{\la}{\lambda}

\newcommand{\br}{\vspace{3 mm}}

\newcommand{\rest}{\upharpoonright}

\newcommand{\Aut}{{\rm{Aut\,}}}

\newcommand{\Id}{{\rm{Id}}}

\mathchardef\mhyphen="2D

\swapnumbers
\theoremstyle{plain}
\newtheorem{thm}{Theorem}[section]
\newtheorem{cor}[thm]{Corollary}
\newtheorem{lem}[thm]{Lemma}
\newtheorem{prop}[thm]{Proposition}

\theoremstyle{definition}
\newtheorem{defn}[thm]{Definition}

\numberwithin{thm}{section}

\usepackage[pdftex,bookmarks,colorlinks,breaklinks]{hyperref}

\begin{document}

\title[Relative weak mixing is generic]
{Relative weak mixing is generic}

\author{Eli Glasner}

\author{Benjamin Weiss}

\address{Department of Mathematics\\
     Tel Aviv University\\
         Tel Aviv\\
         Israel}
\email{glasner@math.tau.ac.il}

\address {Institute of Mathematics\\
 Hebrew University of Jerusalem\\
Jerusalem\\
 Israel}
\email{weiss@math.huji.ac.il}

\date{July 4, 2018}

\begin{abstract}
A classical result of Halmos asserts that among measure preserving transformations
the weak mixing property is generic.
We extend Halmos' result to the collection of ergodic extensions of a fixed, but arbitrary,
aperiodic transformation $T_0$. We then use a result of Ornstein and Weiss 
to extend this relative theorem to the general
(countable) amenable group.
\end{abstract}







\maketitle

\section*{Introduction}
In chapter 19 of Halmos' seminal book \cite{Ha-56}, entitled ``Category", he proves what he calls
the ``second category theorem" (the original proof was published in \cite{Ha-44}):
\begin{thm}
In the weak topology the set of all weakly mixing transformations is an everywhere dense $G_\delta$.
\end{thm}

Only later researchers in ergodic theory got interested in properties of {\bf extensions} of 
measure preserving tr­ansformations and their structure (see e.g. 
\cite{Th-75}, \cite{Z-76a, Z-76b} and \cite{Fur-77}, \cite{Ru-79}). 
Surprisingly, as far as we know, the natural question, whether Halmos' theorem holds for extensions,
did not receive any consideration. The recent result of M. Schnurr \cite{Sch-17} 
(following a question of Tao) caught our attention and in the present note we
extend Halmos' result to the collection of 
measure preserving transformations which are extensions of a fixed, but arbitrary,
aperiodic transformation $T_0$.

Let $\nu$ denote the Lebesgue measure on the unit interval $I$.
Let $\Aut(\nu)$ denote the Polish group of measure preserving automorphisms
of the standard Lebesgue space $(I, \mathcal{B}_0, \nu)$.
We let $\mu = \nu \times \nu$ be the product measure on the space $X = I \times I$.
Sometimes we also write $(X, \mathcal{B}, \mu) = (Z, \mathcal{B}_0, \nu) \times 
(Y, \mathcal{B}_1, \nu)$, where $Z = I =Y$ and $\mathcal{B}$ and $\mathcal{B}_1$ are
the corresponding algebras of measurable sets.

Let $T_0 \in \Aut(\nu)$  be a fixed aperiodic transformation.
On $(X,\mathcal{B}, \mu)$ let $\mathcal{A}$ denote the $\sig$-algebra of sets of the form $A \times I,
\ A \in \mathcal{B}_0$.
Let 
$$
\mathcal{M} = \{T \in \Aut(\mu) : T(A \times I) = T_0A \times I\ {\text{for all}} \ A \in \mathcal{B}_0\},
$$
a closed subset of the Polish group $\Aut(\mu)$.
Denote by $G$ the closed subgroup of $\Aut(\mu)$ consisting of transformations
$S \in \Aut(\mu)$ such that $S \rest \mathcal{A}   = \Id$. 
Clearly $\mathcal{M}$ is invariant under conjugation by elements of $G$.

Let 
$$
X \underset{Z}{\times} X =
\{((z,y),(z,y')) : z, y, y' \in I\}
$$
and let 
$$
\la  = \int_{z \in I} (\del_z \times \nu) \times (\del_z \times \nu) \, d\nu(z),
$$
denote the {\em relatively independent product measure} $\mu \underset{Z}{\times} \mu$,
which is supported on $X \underset{Z}{\times} X$.

\begin{defn}
An element $T \in \mathcal{M}$ is {\bf relatively weakly mixing over} $T_0$ when 
the measure $\la$ is $T \times T$-ergodic.
\end{defn}

We can now state our main theorem.

\begin{thm}\label{main}
The collection $\mathcal{W}$ 
of measure preserving transformations in $\mathcal{M}$, which are weakly mixing relative 
to $T_0$, is a dense $G_\del$ subset of $\mathcal{M}$.
\end{thm}

In section \ref{sec-Z} we prove Theorem \ref{main}.
In section \ref{sec-A} we apply a general argument regarding extensions, originated in
\cite{RW-00}, and then use a result of Ornstein and Weiss \cite{OW-80} (for proofs see
\cite{CFW-81}), to extend the relative theorem to the general
(countable) amenable group.

We thank J.-P. Thouvenot and Michael Schnurr for suggesting the present formulations of 
Proposition \ref{dense} and Theorem \ref{main}, which are now slightly stronger than the original ones
(see \cite{GW-18}).

\section{The relative theorem for $\mathbb{Z}$-actions}\label{sec-Z}

%

We first note that, by a theorem of Rohlin, we can write every $T \in \mathcal{M}$ as a skew product over $T_0$:
$$
T(z,y) = (T_0z, \tau_z(y)),
$$
where $z \mapsto \tau_z$ is a measurable map from $Z$ into the Polish group $\Aut(\nu)$
(see e.g. \cite[Theorem 3.18]{Gl-03}).

\begin{defn}
Given a measurable partition $\{A_0, A_1, \dots, A_k\}$ of $I$
and a finite set $\{R_0, R_1,\dots,R_k\}$ of elements of $\Aut(\nu)$,
define a {\bf piecewise constant skew product over} $T_0$, $R \in 
\mathcal{M} \subset \Aut(\mu)$ by
the formula 
$$
R(z,y) = (T_0z, \rho_z(y)),
$$ 
where for each $j$, $\rho_z$ is the constant transformation $R_j$
on the cell $A_j$. 
\end{defn}

\begin{lem}\label{pw-dense}
The collection of piecewise constant skew products over $T_0$ is a dense subset of $\mathcal{M}$.
\end{lem}

\begin{proof}
Let $d$ be a compatible metric on $\Aut(\nu)$. 
An arbitrary element $T \in \mathcal{M}$ has the form
$$
T(z, y) = (T_0z, \tau_zy),
$$
where $\tau : [0,1] \to \Aut(\nu)$ is a measurable mapping. 
Denote by $\tilde{\nu}$ the push-forward measure $\tilde{\nu}=\tau_*\nu = \nu \circ \tau^{-1}$.
This is a regular measure on $\Aut(\nu)$ hence, given $\ep >0$, there is a compact set 
$K \subset \Aut(\nu)$ with $\tilde{\nu}(K) > 1-\ep$.
Let $\{C_i\}_{i=1}^k$ be a partition of $K$ into sets with diameter $< \ep$
and set $C_0 = \Aut(\nu) \setminus K$.
Let $A_i = \tau^{-1}(C_i), \ 0 \le i \le k$.
Define $R_0 = \Id$ and for each $1 \le i \le k$ choose some $R_i \in C_i$.
Finally let $R = R_{\{A_i, R_i\}}$ be defined by 
$$
R(z,y) = (T_0z, \rho_z(y)),
$$ 
where for each $i$, $\rho_z$ is the constant transformation $R_i$
on the cell $A_i$. Clearly $T \overset{\ep}{\sim} R$.
\end{proof}

We can now prove a relative analogue of Halmos' conjugation theorem for aperiodic transformations:

\begin{prop}\label{dense}
For every element $T \in \mathcal{M}$, its orbit under conjugations with elements of $G$ is dense in $\mathcal{M}$.
\end{prop}

\begin{proof}
We write $T$ as a skew product over $T_0$:
$$
T(z,y) = (T_0z, \tau_z(y)).
$$

By  Lemma \ref{pw-dense} it suffices to show that given
$R = R_{\{A_j, R_j\}}$, a piecewise constant skew product over $T_0$, and $\ep > 0$,
there is some $Q \in G$ such that, with  
$V = Q T Q^{-1}$ we have
$V \overset{\ep }{\sim} R$.

Fix $n$ so that $\frac{1}{n} < \ep/2$.
Let $\{B, T_0B, \dots, T^{n-1}B\}$ be a Rohlin tower for $T_0$ with
$\nu(I \setminus \bigcup_{i=0}^{n-1}T_0^i B) < \ep/2$
(see e.g. \cite{Ha-56}).


We refine the tower with respect to the partition $\{A_j : 1 \le j \le k\}$. 
This partitions $B$ into subsets $\{B_l : 1 \le l \le L\}$,
such that for each $l$ and each $i$ there is some $\al(l,i)$ with $T^iB_l \subset A_{\al(l,i)}$.

Next, define $Q \in G$,
$$
Q(z,y) = (z, \kappa_z(y)),
$$ 
as follows:

\begin{itemize}
\item
For each $l$, 
set ${\kappa^{(l)}}_z = \Id$ on $B_l$, 
\item 
for $i= 1,\dots, n -2$, let ${\kappa^{(l)}}_z \in \Aut(\nu)$ satisfy 
${\kappa^{(l)}}_{T_0z} \tau_z ({{\kappa^{(l)}}_z})^{-1} = R_{\al(l,i)}$, for $z \in T_0^i B_l,\ i =0,1,\dots, n-2$, 
\item
${\kappa^{(l)}}_{T_0 z} \tau_z  ({{{\kappa^{(l)}}}_z})^{-1} = R_0$ for $z \in T_0^{n-1} B_l$.
\end{itemize}

Thus
\begin{gather*}
\kappa^{(l)}_z  = \Id, \quad z \in B_l \\
\kappa^{(l)}_{T_0^{i+1}z} =  R_{\al(l,i)} {\kappa^{(l)}}_{T_0^{i}z} {(\tau_{T_0^i z})}^{-1}, \quad
z \in B_l, \ i=1, \dots n-1,
\end{gather*}
so that $Q$ is defined by
\begin{gather*}
Q(z, y) = (z, \kappa^{-1}_z y) \quad {\text{for}}\quad  z \in T_0^i B_l, \quad 0 \le i \le n-1,
\quad 1 \le l \le L \\
Q(z, y) = (z, y), \quad {\text{for}} \quad z \in I \setminus \bigcup_{i=0}^{n-1}T_0^i B.
\end{gather*}
Set $V = Q T Q^{-1}$.
It is now easy to check that $V \overset{\ep }{\sim} R$, and we are done.
\end{proof}

%

\begin{lem}\label{wm-Gdel}
The collection $\mathcal{W}$ consisting of transformations which are relatively 
weakly mixing over $T_0$ is a $G_\del$ subset of $\mathcal{M}$.
\end{lem}

\begin{proof}

Given $A, B \in \mathcal{X}$, the $\sig$-algebra of measurable subsets on $X$,
$\ep >0$ and $N \in \N$, set
$$
\mathcal{E}_{A, B, \ep, N} =
\left\{ T \in \mathcal{M} :
\left\| \frac{1}{N} \sum_{n =0}^{N-1} (T^n \times T^n)(\ch_A \otimes \ch_B) -
\mathbb{E}(\mathbb{E}(\ch_A | Z)\mathbb{E}(\ch_B  | Z)) \right\|_{L^2(\la)} < \ep
\right\}.
$$
Clearly each $\mathcal{E}_{A, B, \ep, N}$ is an open set and  
$$
\mathcal{W} = \bigcap_{i, j}\bigcap_k \bigcup_N \mathcal{E}_{A_i, A_j, \frac1k, N}, 
$$
where $\{A_i\}_{i \in \N}$ is a countable dense collection in $\mathcal{X}$.
\end{proof}

%
%
%

We are now in a position to complete the proof of our main theorem.

\begin{proof}[Proof of Theorem \ref{main}]
Let $W_0 \in \Aut(\nu)$ be a fixed weakly mixing transformation
and set $W = T_0 \times W_0$ on $I \times I$.
Note that $W$ is a weakly mixing extension of $T_0$.

By Proposition \ref{dense} the $G$-orbit of $W$ is dense in $\mathcal{M}$.
Since every element of this orbit is weakly mixing relative to $T_0$ (being a conjugate of 
$W = T_0 \times W_0$), and as the collection of
all the transformations which are weakly mixing relative to $T_0$ is a $G_\delta$ subset of $\mathcal{M}$
(by Lemma \ref{wm-Gdel}),
the assertion of the theorem follows.
\end{proof}

\br

Let
$$
\mathcal{N} = \{T \in \Aut(\mu) : T \ {\text{leaves the $\sig$-algebra $\mathcal{A}$ invariant}}\},
$$
and let 
$$
\widetilde{\mathcal{W}} = \{T \in \mathcal{N} : T \ {\text{is a weakly mixing extension of}}\
T' = T \rest \mathcal{A}\}.
$$
As the collection of ergodic transformations is a dense $G_\del$ subset of $\Aut(\nu)$,
Theorem \ref{main}, combined with the Kuratowski-Ulam theorem 
\cite[Theorem 15.4]{Ox-80}, immediately yield
the following result of M. Schnurr \cite{Sch-17}:

\begin{cor}
The collection $\widetilde{\mathcal{W}}$ forms a residual subset of $\mathcal{N}$.
\end{cor}

\br

\section{The relative theorem for amenable groups}\label{sec-A}
The result of Halmos, that weak mixing is generic in $\Aut(\nu)$, 
can be extended to any countable amenable group as soon as one establishes that orbits under
conjugation by $\Aut(\nu)$ of free actions are dense. 
In turn, this relies on the
Rohlin lemma (in fact Halmos' original proof used a weaker version
of the lemma which was first stated by Rohlin \cite{Ro-48}). Such a result 
for amenable groups was established in \cite{OW-80}.
It is therefore natural to ask about the analogue of theorem \ref{main}
for actions of an amenable group. The fact that,
for a given amenable group $G$, the 
collection of the relatively weakly mixing extensions of a fixed action forms a 
$G_{\delta}$ can be proved just like we proved lemma \ref{wm-Gdel}.
It thus remains to show that  this collection is dense.
In order to do this we proceed as follows.

\begin{defn}
If $G$ and $H$ are two countable groups acting as measure preserving transformations 
$\{T_g\}_{g \in G}, \{S_h\}_{h \in H}$
on the measure space $(Z,\nu)$ we say that the actions are orbit equivalent
if for $\nu$-a.e. $z \in Z, \  Gz = Hz$.
\end{defn}

In \cite{OW-80} and \cite{CFW-81} it is shown that any ergodic measure preserving  action of an 
amenable group is orbit equivalent to an action of $\Z$.

\br

We fix an arbitrary countable amenable group $G$.
As in section \ref{sec-Z} we let $\nu$ denote the Lebesgue measure on the unit interval $I$,
and let $\mu = \nu \times \nu$ be the product measure on the space $X = I \times I$.
Again we write $(X, \mathcal{B},  \mu) = (Z, \mathcal{B}_0, \nu) \times
(Y, \mathcal{B_1},\nu)$, where $Z = I =Y$.

Let $\mathbb{A}(G,\nu)$ denote the Polish space of measure preserving actions $\{T_g\}_{g \in G}$ of 
$G$ on $(Z, \nu)$.
Let $\{(T_0)_g\}_{g \in G} \in \mathbb{A}(G, \nu)$ be a fixed ergodic $G$-action.
On $(X, \mathcal{B},\mu)$ let $\mathcal{A}(\nu)$ denote the $\sig$-algebra of sets of the form 
$A \times I, \ A \in \mathcal{B}_0$.
Let 
$$
\mathcal{M} = \{T \in \mathbb{A}(G, \mu) : T_g(A \times I) = (T_0)_g A \times I\ {\text{for all}} \ A \in \mathcal{B}_0 \ {\text{and}}\ g \in G\},
$$
a closed subset of the Polish space $\mathbb{A}(G, \mu)$.

Next we remark that if two actions $\{T_g\}$ of $G$  and $\{S_h\}$ of $H$  are orbit equivalent, the 
$T$ action is ergodic if and only if the $S$ action is ergodic. The acting groups may be quite different.

If $S_0$ is an action of a group $H$ on  $(Z, \mathcal{B}_0, \nu)$ that is orbit equivalent to $T_0$,
then the collection of $H$-actions $\mathcal{M}'$ defined by
$$
\mathcal{M}' 
= \{S \in \mathbb{A}(H, \mu) : S_h(A \times I) = (S_0)_h A \times I\ {\text{for all}} \ A \in \mathcal{B}_0 \ {\text{and}}\ h \in H\},
$$
can be matched, via a canonical bijective map, with the collection $\mathcal{M}$ as follows.
An element $T  \in \mathcal{M}$ is determined uniquely by the cocycle $\tau(z,g)$ with values in
$\Aut(\nu)$
$$
T_g(z,y) = ((T_0)_g z, \tau(z, g) y), \quad g \in G.
$$
Such cocycles depend only on the equivalence relation defined on $Z$ by the $G$-action $T_0$.
Since $S_0$ is orbit equivalent to $T_0$, i.e.
they have the same orbits, each such cocycle also defines a unique element $S \in \mathcal{M}'$.
Clearly $T$ is weakly mixing over $T_0$ iff $S$ is weakly mixing over $S_0$.
%
%
\begin{thm}\label{main-amen}
Let $G$ be a countable amenable group.
The collection $\mathcal{W}$ 
of measure preserving actions in $\mathcal{M}$ which are weakly mixing relative 
to  a given ergodic action $\{(T_0)_g\}_{g \in G}$ is a dense $G_\del$ subset of $\mathcal{M}$.
\end{thm}


 \begin{proof}
Apply the theorem from \cite{OW-80} and \cite{CFW-81} to get a measure preserving transformation 
$R \in \Aut(\nu)$ that is orbit equivalent to the $G$-action $\{(T_0)_{g}\}_{g \in G}$. 
Now check the relative weak mixing using this transformation.

More explicitly, when the equivalence relation is fixed, as above, 
the extensions of $R$ are in one to one correspondence with 
cocycles of the $T_0$ $G$-action with values in the group of measure preserving transformations of the fiber $(Y, \mathcal{B}_1, \nu)$. Thus theorem \ref{main}, when applied to this $\Z$-action, 
gives the desired result.
\end{proof}


\end{document}